 \documentclass[final,1p,times]{elsarticle}

\usepackage{epsfig}
\usepackage{epstopdf}
\usepackage{subfigure}

\usepackage{amssymb}
\usepackage{amsthm}
\usepackage{amsmath}

\newtheorem{theorem}{Theorem}[section] 
\newtheorem{lemma}[theorem]{Lemma}     

\numberwithin{equation}{section}

\journal{Constructive Approximation}

\begin{document}

\begin{frontmatter}


\author{Patrick Arthur Miller} 
\ead{patrickarthurmiller@gmail.com}
\address{Mathematics Department, The Graduate Center -- City University of New York, 365~5th~Avenue, New York, NY 10016}
\address{Department of Mathematics \& Computer Science, Rutgers University, 101 Warren St., Newark, NJ 07102} 

\title{Geometric Multiproducts:  A New Extrapolation Tool}


\author{}

\address{}

\begin{abstract}
We show how to extrapolate an analytic function (or a smooth signal) by multiplying and dividing its values on geometric sequences that collapse to a point.
\end{abstract}


\end{frontmatter}


\section{Introduction}

 Digital signal processing normally uses signal samples that are uniformly spaced in time; the sample measurement times form an arithmetic sequence \cite{dsp}.  This article shows how to extrapolate an analytic function (or a smooth signal), by multiplying and dividing function values (or signal samples) at points that form sets of geometric sequences.

The starting point for this work was a trigonometric identity attributed to Euler, written in a particular way:
\begin{equation}
\label{eq:Euler}
\prod_{n=1}^\infty\,\cos\left(\frac{x}{2^n}\right) = \frac{1}{x} \int_0^x\,\cos(s)\,ds.
\end{equation}

Eqation~\ref{eq:Euler} possesses a certain parallelism.  In forming the simple integral on the right, the cosine function is sampled on a dense arithmetic sequence, the values are added, and the sum is normalized.  On the left, the cosine function (already normalized to unity at the origin) is sampled  on a geometric sequence, and the values are multiplied.  The equation also has a predictive capability:  on the right we have the average value of cosine on the interval $[0, x]$, whereas on the left we evaluate cosine only in the \emph{first half} of the interval.  Lastly, since the equation has only one function present, one might wonder whether there is a generalization for other functions.  This prompted us to look for a similar equation, which we hoped would predict the value of a function at one point by multiplying prior geometric values.

The solution we found does indeed multiply function values from geometric sampling.  But the function is evaluated on multiple geometric sequences. And, strengthening the parallelism, the values are multiplied as the geometric sequences become dense by collapsing to a point.

\section{An Extrapolation Tool}

We call our result the Geometric Multiproduct.

We let $x$ represent the point to which we want to extrapolate the function $f$.  To assure convergence of infinite products \cite{infprod}, we require that $f$ be normalized to unity at the origin and that the geometric ratio $(1/r)$ be less than one.
Our result is expressed in the following theorem.
\begin{theorem}
\label{thm:GMP}
Suppose that $f\textrm{:\ }  \mathbb{R}\rightarrow\mathbb{R}$ is analytic and nonzero on the interval $[0, x]$, that $f(0)=1$, and that $r>1$.
Let $\mathcal{S}$ be a set of positive integers, $\mathcal{S} \subset \mathbb{N}^*$.   If we define
\begin{equation}
\label{eq:places}
x_n\left(\mathcal{S},r,x\right) \,=\, \left[ \,\prod_{k\in \mathcal{S}}\,(r^k-1)^{1/k}\right]\cdot\frac{x}{r^n}\,\,\,\,
\end{equation}
\begin{equation}
\label{eq:factors}
and\,\,\,\,\,\,\,\,\,\,\,\mathcal{P}_f\left(\mathcal{S},r,x\right) \,=\,\,\,\,\prod_{n=|\mathcal{S}|}^\infty \,\,\left[\,f\left(x_n(\mathcal{S},r,x\right))\,\right]\,^{^{\textrm{\large{$\binom{n-1}{|\mathcal{S}|-1}$}}}}\,,
\end{equation}
\begin{equation}
\label{eq:powerproduct}
then\,\,\,\,\,\,\,\,\,\,f(x) \,= \,\, \lim_{r\downarrow1}\,\,\frac{\ \ \prod\limits_{{|\mathcal{S}|\, odd}}\,\,\mathcal{P}_f\left(\mathcal{S}, r,x\right)\ } {\,\prod\limits_{{|\mathcal{S}|\, even}}^{\ }\,\mathcal{P}_f\left(\mathcal{S}, r,x\right)}\,.
\end{equation}
\end{theorem}

\subsection*{Discussion}

\begin{itemize}
\item Equation~\ref{eq:places} defines the geometric sequences  $\{x_n\}_{n=1}^\infty$.  There is a geometric sequence $\{x_n\}$ for each finite subset $\mathcal{S}$ of the natural numbers $\mathbb{N}^*$ (positive integers).  The points $\{x_n\}$ depend upon $x$ and the geometric ratio $r>1$, but also have a coefficient that depends upon the integers in the set $\mathcal{S}$.  For simplicity, we have assumed that the geometric sequences converge to the origin.
\item Equation~\ref{eq:factors} uses the function $f$  to calculate a product $\mathcal{P}_f$ for each set $\mathcal{S}$. \  $\mathcal{P}_f$  is obtained by multiplying values of the function $f$, evaluated on the geometric sequence for the set $\mathcal{S}$.  The index $n$ starts at $|\mathcal{S}|$, not $1$, and each factor of $f$ has a multiplicity $\binom{n-1}{|\mathcal{S}|-1}$. These details will  be explained later.
\item Equation~\ref{eq:powerproduct} calculates the value of the function $f$ at the point $x$ by multiplying and dividing factors of $\mathcal{P}_f$.  There is one factor $\mathcal{P}_f$ in the numerator for each subset $\mathcal{S}$ of $\mathbb{N}^*$ containing an odd number of integers, and one factor $\mathcal{P}_f$ in the denominator for each subset $\mathcal{S}$ of $\mathbb{N}^*$ containing an even number of integers.  After multiplying and dividing factors of $\mathcal{P}_f$, we take the limit as $r$ approaches $1$ from above, $r\downarrow1$.
\end{itemize}

Notice that all points $\{x_n\}$ in all geometric sequences aproach zero as $r\downarrow1$.  We shall prove that, although $f(0)=1$, the formula produces $f(x)$, not 1.

\section{Two Simple Examples}

Before proving our result, we show how to use the Geometric Multiproduct by extrapolating two simple functions.  To make the calculations manageable, we must truncate the infnite products and pick a value for $r$ near $1$.  A truncated Geometric Multiproduct gives us an \emph{estimate} for the value of a function.
\begin{figure}[h!]
\center{
\includegraphics[scale=0.3]{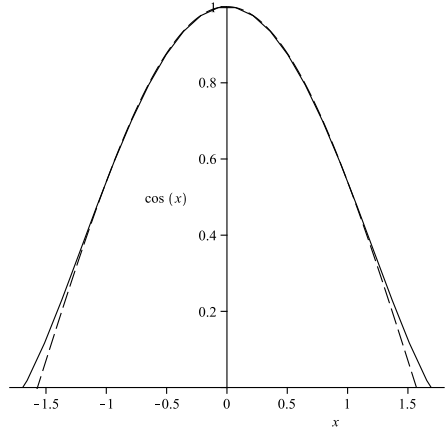}}
\caption{Geometric Multiroducpt for $\cos(x)$.  Solid line is the estimate; dashed line is the function.  The estimate has $\mathcal{S}_{max}=\{2,4\},\, n_{max}=10,$ and $\, r=\sqrt{2}$.}
\label{figexcos}
\end{figure}

\begin{figure}[h!]
\begin{center}
\includegraphics[scale=.3]{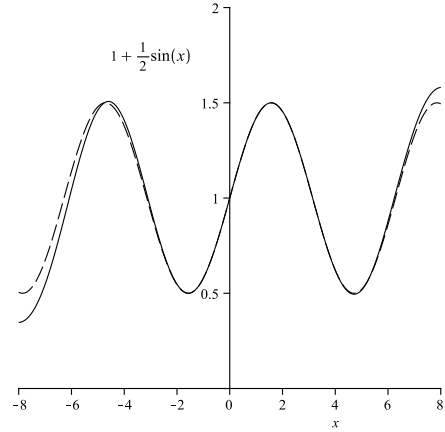}
\caption{Geometric Multiproduct estimate for $1+\frac{1}{2}sin(x)$.  Solid line is the estimate; dashed line is the function.  The estimate has $\mathcal{S}_{max} = \{1,2,3,4\},\, n_{max}=40,$ and  $r=2$.}
\label{halfsin-nonanal}
\end{center}
\end{figure}
We first perform an extrapolation for $\cos(x)$.  For this example, we let the index $n$ run only from $1$ to $n_{max}=10$.
 We truncate the sets $\{\mathcal{S}\}$ at $\mathcal{S}_{max} = \{2,4\}$. This means  we use the power set of $\mathcal{S}_{max}: \, \{2\}, \{4\}$, and $\{2, 4\}$.  (Later we will show why we only need to use even integers for even functions.)   This gives
\begin{align}
\cos(x) &\approx \lim_{r\approx1^+}\,\mathcal{P}_{\cos}\left(\{2\}, r,x\right) \cdot \frac{\mathcal{P}_{\cos}\left(\{4\}, r,x\right)}{\mathcal{P}_{\cos}\left(\{2,4\}, r,x\right)} \nonumber\  \\ 
&\approx \lim_{r\approx1^+}\,\left\{\prod_{n=1}^{10}\cos\left[(r^2-1)^{1/2}\,x/r^n\right]\right\} \ \cdot \frac{\prod_{n=1}^{10}\cos\left[(r^4-1)^{1/4}\,x/r^n\right]}{\prod_{n=2}^{10}\,\left\{\cos\left[(r^2-1)^{1/2}(r^4-1)^{1/4}\,x/r^n\right]\right\}^{\,\,n-1}} \nonumber
\end{align}

      For accuracy, we should truncate the geometric sequences of cosine arguments near their limit point of zero.  This is satisfied if $r^{10}$ is large.  To get $r^{10} = 2^5$, we set $r=\sqrt{2}$.  Graphing both sides, we obtain the result shown in Fig. \ref{figexcos}.  The agreement between the Geometric Multiproduct estimate and $\cos(x)$ is striking, given how drastically we truncated the infinite products and how far $r=\sqrt{2}$ is from 1.

Figure \ref{halfsin-nonanal} shows an extrapolation for the function $1+\frac{1}{2}\sin(x)$.  If we wanted to extrapolate $\sin(x)$, we might use this function instead, to assure that $f(0)=1$ and $f>0$.  The extrapolation has $\mathcal{S}_{max}=\{1,2,3,4\},\ n_{max}=40, \,\textrm{and\ } r=2$.

Increasing $\mathcal{S}_{max}$ would give more oscillations to the estimate for this function.  In this estimate, calculation of $f$ at one point involves multiplication of more than 75,000 factors of earlier values of $f$.  We see again that Geometric Multiproducts are highly tolerant of $r$-values that are quite far from $1$, at least in percentage terms.

\section{Proof of the Geometric Multiproduct}
\begin{proof} We divide the proof of Theorem \ref{thm:GMP} into five short lemmas.
\subsection{An Invariance}
We first define a geometric sequence and a corresponding function that have a very useful property.
\begin{lemma}
Let $\{ x_n \}_{n=1}^\infty$ be the geometric sequence $x_n(k,r)  = (r^k-1)^{1/k}\,x/r^n$, with $r>1$.  Let $f_k = \exp(c_k x^k)$, where $c_k$ is a constant and $k$ is an integer.  Then
\begin{equation}
\prod_{n=1}^\infty f_k(x_n) = f_k(x).
\end{equation}
\end{lemma}
Here $1/r$ is the geometric ratio, and $x$ is the point at which we want to calculate the value of $f$.
Loosely speaking, the function $f_k(x)$ is invariant under multiplication of function values on the sequence  $x_n(k,r)$.

\begin{proof}
The proof is straightforward algebra:
\begin{align}
\prod_{n=1}^\infty\, f_k\left(\frac{(r^k-1)^{1/k}\, x}{r^n}\right) \nonumber&= \prod_{n=1}^\infty\, \textrm{exp}\left[c_k \left(\frac{(r^k-1)^{1/k}\, x}{r^n}\right)^k\right] \nonumber  
= \textrm{exp}\left[\sum_{n=1}^\infty\, c_kx^k(r^k-1)\left(\frac{1}{r^{nk}}\right)\right] \nonumber \\ 
&= \textrm{exp}\left[c_kx^k(r^k-1)\sum_{n=1}^\infty\, \left(\frac{1}{r^k}\right)^n\right] \nonumber  
= \textrm{exp}\left[c_kx^k(r^k-1)\frac{1/r^k}{1-1/r^k}\right] \nonumber \\ 
&= \textrm{exp}(c_kx^k) = f_k(x). \nonumber
\end{align}
\end{proof}
Since $f_k(x)$  is independent of $r$, we may take the limit $r\downarrow1$ on the left-hand side of the equation.
\begin{equation}
\label{eq:liminvar}
\lim\limits_{r \downarrow 1}\,\,\prod_{n=1}^\infty\, f_k\left(\frac{(r^k-1)^{1/k}\, x}{r^n}\right) = f_k(x).
\end{equation}
The limit $r\downarrow1$ makes the geometric sequence \emph{dense}, because as $r\downarrow1$ all points of the geometric sequence approach the origin:  $\lim\limits_{r\downarrow1}\,(r^k-1)^{1/k}\,x/r^n = 0$.  Eq.~\ref{eq:liminvar} says we can multiply values of $f_k$ evaluated on a geometric sequence within an infinitesimal neighborhood of zero and, although $f_k(0)=1$, we do \emph{not} get 1.  We get the value of $f_k$ at the point $x$.

\subsection{Component Factors of a Function}

We write a normalized function $f$ as a product of $f_k$'s.  If $f\textrm{:\ }  \mathbb{C}\rightarrow\mathbb{C}$ is holomorphic and nonzero in some simply-connected open neighborhood, therein $\log(f)$ is also holomorphic and can be written as a power series \cite{analyticlog}:
\begin{equation}
\label{Taylog}
f(x) = e^{log(f)} = \exp(\sum_{k=1}^\infty c_k\,x^k) = \prod_{k=1}^\infty f_k(x),
\end{equation}
where, as before, $f_k(x)=\exp(c_k x^k)$ with $k$ an integer.  By restriction, the statement is also true for a positive real-analytic function $f$ on open interval of the real axis. 

We call the factor $f_k(x)$ the $k$-th \emph{component} of the function $f$.  For an even function, only even values of the index $k$ would appear.  (This is why only even integers occur in the sets $\{\mathcal{S}\}$ for an even function.)

 It would be convenient if we could produce $f_k$ merely by geometrically sampling $f$ on the order-$k$ sequence.  However, when we sample $f$, we sample not only its $f_k$ component, but also all of the other components of $f$. So we need to understand what happens if we sample the $j$-th component of $f$ on the order-$k$ sequence.
\begin{lemma}  Let $f_j(x)=\exp(c_j\,x^j)$.  Then
\begin{align}
\label{eq:pollution}
\lim_{r \downarrow 1} \prod_{n = 1}^\infty\, f_j \left(\frac{(r^k-1)^{1/k}\, x}{r^n}\right) &= 
\begin{cases}
f_k(x)	&	\textrm{if  $j = k$}\\ 
1 & \textrm{if  $j>k$ \textrm{or} $c_j = 0$} \\
0 & \textrm{if  $j<k$ \textrm{and} $c_j<0$} \\
\infty & \textrm{if $j<k$ \textrm{and} $c_j>0$} 
\end{cases}
\end{align}
\end{lemma}
\begin{proof}
\begin{align*}
\  \prod_{n = 1}^\infty\, f_j \left(\frac{(r^k-1)^{1/k}\, x}{r^n}\right)  &\,=\, \prod_{n=1}^\infty\, \textrm{exp}\left[c_j\left(\frac{(r^k-1)^{1/k}\, x}{r^n}\right)^j\right] \\ 
&\,=\, \textrm{exp}\left[c_jx^j (r^k-1)^{j/k} \sum_{n=1}^\infty\, \left(\frac{1}{r^j}\right)^n\right]  \\
&\,=\, \textrm{exp}\left[c_jx^j\frac{(r^k-1)^{j/k}}{r^j-1}\right] . \nonumber
\end{align*}
Letting $ r = 1+\delta$, using the binomial theorem, keeping lowest order terms in $\delta$, and letting $r \downarrow 1$ (i.e., $\delta \downarrow 0$), we obtain
\begin{align}
\lim_{r \downarrow 1}\,\, \prod_{n = 1}^\infty\, f_j \left(\frac{(r^k-1)^{1/k}\, x}{r^n}\right) &= \lim_{\delta \downarrow 0}\,\, \textrm{exp}\left(c_jx^j\frac{[(1+\delta)^k-1]^{j/k}}{(1+\delta)^j-1)}\right) \nonumber \\ 
&= \lim_{\delta \downarrow 0}\,\,\textrm{exp}\left(c_j\,x^j\, \frac{[1 + k\delta + O(\delta^2) - 1]^{j/k}}{(1 + j\delta + O(\delta^2) - 1)}\right) \nonumber \\ 
&=  \lim_{\delta \downarrow 0}\,\,\textrm{exp}\left(c_j x^j\,\cdot\,\frac{k^{j/k}}{j}\,\cdot\, \lim_{\delta \downarrow 0}\,\delta^{(j/k) - 1}\right) \nonumber
\end{align}
The cases in Eq.~\ref{eq:pollution} follow directly.
\end{proof}

\subsection{A Key Intermediate Result}

According to Eq. \ref{eq:pollution}, taking $r\downarrow1$ (i.e., $\delta \downarrow 0$) eliminates the effect of components $f_j$ with $j>k$ from the order-$k$ geometric sampling of $f$:  they each provide a factor of 1.  The effect of the finitely-many components with $j<k$ can be divided out before the limit is taken.

\begin{lemma}
 Let $f_k(x)=\exp(c_j\,x^k)$.  Then
\begin{equation}
\label{eq:fk}
f_k(x) = \lim_{r\downarrow1}\,\,\prod_{n=1}^\infty  \ \left[f\left(\frac{(r^k-1)^{1/k}\,\, x}{r^n}\right)\,\,\textrm{\LARGE{/}}\,\prod\limits_{i<k}\,f_i\left(\frac{(r^k-1)^{1/k}\,\, x}{r^n}\right)\,\right] \ .
\end{equation}
\end{lemma}
\begin{proof}To express $f_k$ in terms of $f$, we first augment $f_k$ with the other components necessary to construct $f$.
\begin{align}
\label{eq:fkbig}
&f_k(x) = \prod_{n=1}^\infty\,f_k\left(\frac{(R^k-1)^{1/k}\,\,x}{R^n}\right) \nonumber  \\
&= \prod_{n=1}^\infty\,\,\left[\prod\limits_{i<k}\,f_i\left(\frac{(R^k-1)^{1/k}\,\, x}{R^n}\right)\right]\,f_k\left(\frac{(R^k-1)^{1/k}\,\, x}{R^n}\right)\,\left[\lim\limits_{r\downarrow1}\,\,\prod\limits_{i>k}f_i\left(\frac{(r^k-1)^{1/k}\,\, x}{r^n}\right)\right]\,\,\textrm{\LARGE{/}}\prod\limits_{i<k}\,f_i\left(\frac{(R^k-1)^{1/k}\,\, x}{R^n}\right)
\end{align}
This equation is true for any value of R.  The last factor in brackets in the numerator is unity.  If we set $R = r$, we can take the limit of the entire righthand side.  This gives Eq.~\ref{eq:fk}.
\end{proof}
Note that, as empasized in equation~\ref{eq:fkbig}, it is only the higher-order components ($j>k$) that require us to take the limit $r\downarrow1$.  Since we expect the high-order components to be less significant than low-order ones, we might also expect our results to depend only weakly upon taking the limit $r\downarrow1$.  Numerical examples show that we do not need to use an $r$-value \emph{very} close to 1.

\subsection{Calculating the Components of a Function}

We want to express the value of $f$ at $x$ solely in terms of earlier values of $f$ itself.  We can calculate $f(x)$ if we can calculate its components $f_k(x)$.   To calculate the  denominator of $f_k$ in Eq.~\eqref{eq:fk}, we must know how to calculate $f_i$ for all $i<k$.  But where does $f_i$ come from?  Answer:  the formula for $f_k$ can be used recursively to calculate its own denominator.

The first component has no denominator and can be calculated solely from $f$:
\begin{equation}
f_1(x) = \lim_{r\downarrow1}\ \ \prod_{n_1=1}^\infty f\left[\frac{(r-1)\,x}{r^{n_1}}\right]. \nonumber
\end{equation}
  The second component $f_2$ will have $f$ in its numerator and $f_1$ in its denominator, but $f_1$ can be expressed in terms of $f$:

\begin{align*}
f_2(x) &= \lim_{r\downarrow1}\ \ \prod_{n_2=1}^\infty \,\,f\left[\frac{\sqrt{r^2-1}\,\, x}{r^{n_2}}\right]\,\,\,\,\textrm{\LARGE{/}}f_1\left[\frac{\sqrt{r^2-1}\,\, x}{r^{n_2}}\right] \\
&= \lim_{r\downarrow1}\ \prod_{n_1,n_2=1}^\infty f\left[\frac{\sqrt{r^2-1}\,\, x}{r^{n_2}}\right]\,\,\textrm{\LARGE{/}}f\left[\frac{(r-1)\sqrt{r^2-1}\,\, x}{r^{n_1+n_2}}\right]. 
\end{align*}

 The number of factors rapidly increases:  $f_3$ will have $f_2$ and $f_1$ in its denominator, but the $f_2$ in the denominator will spawn an $f_1$ in the denominator of the denominator (i.e., in the numerator).  After expressing each $f_1$ in terms of $f$,
\begin{align*}
f_3(x) = \lim_{r\downarrow1}\ \ \prod_{n_1,n_2,n_3=1}^\infty \ f\left[\frac{(r^3-1)^{1/3}\,x}{r^{n_3}}\right]&\ f\left[\frac{(r-1)(r^2-1)^{1/2}(r^3-1)^{1/3}\,x}{r^{n_1+n_2+n_3}}\right] \\
&\,\,\textrm{\LARGE{/}}\ f\left[\frac{(r-1)(r^3-1)^{1/3}\,x}{r^{n_1+n_3}}\right]\ f\left[\frac{(r^2-1)^{1/2}(r^3-1)^{1/3}\,x}{r^{n_2+n_3}}\right]\,.
\end{align*}

There is a shortcut notation that can help us cope with the complexity.   The third component can be expressed as
\begin{equation}
f_3(x) = \lim_{r\downarrow1}\,\prod\,\frac{\,\,f(\{3\} \circ x)\,f(\{1,2,3\} \circ x)}{f(\{1,3\} \circ x)\,f(\{2,3\} \circ x)}, \nonumber
\end{equation}
in which $\{\ldots,\,\, j\,, \ldots\} \circ x$ indicates the presence of  a factor $(r^j-1)^{1/j}/r^{n_j}$ multiplying $x$, and in the product $\prod$ all indices $n_j$ run from 1 to $\infty$.

As $k$ gets larger, we begin to see a pattern.
\begin{itemize}
\item  In the numerator of $f_k$ we have all sets of positive integers whose greatest element is $k$ and have an \emph{odd} number of elements.
\item  In the denominator we have all sets of positive integers whose greatest element is $k$ and have an \emph{even} number of elements.
\end{itemize}
We prove that this pattern always holds.
\begin{lemma}
Let the index set $\{\mathcal{S}_k\}$ be the family of sets of positive integers, each of whose largest element is $k$:  $\,\mathcal{S}_k \subset \mathbb{N}^*\, \textrm{and} \,\max(\mathcal{S}_k) = k$.  Then, with the previous assumptions and notation,
\begin{equation}
\label{eq:fkcirc}
f_k(x) =  \lim_{r\downarrow1}\,\prod\limits_{|\mathcal{S}_k|\ \textrm{odd}}\,\,\,\prod\limits_{\stackrel{\textrm{\small{$\ldots,\,n_j,\ldots=\, 1\, $}}}{j \in \mathcal{S}_k}}^\infty\,f\left(\mathcal{S}_k \circ x\right)\,\,\textrm{\LARGE{/}}\prod\limits_{|\mathcal{S}_k|\ \textrm{even}}\,\,\,\prod\limits_{\stackrel{\textrm{\small{$\ldots,\,n_j ,\ldots=\, 1\,$}}}{j \in \mathcal{S}_k}}^\infty\,f\left(\mathcal{S}_k \circ x\right)\,.
\end{equation}
\end{lemma}
\begin{proof}
The proof uses induction on $k$.  We have shown that the assertion holds for small $k$.  We show that if it holds for all $i<k$, it also holds for $k$.  Eq.~\ref{eq:fk} can be written
\begin{align}
f_k(x) &= \lim_{r \downarrow 1}\,\prod\limits_{n_k = 1}^\infty\,\frac{ f\left(\{k\} \circ x\right)}{\prod\limits_{i<k}\, f_i\left(\{k\} \circ x\right)} \nonumber\\
&=  \lim_{r \downarrow 1}\,\prod\limits_{n_k = 1}^\infty f\left(\{k\} \circ x\right) \nonumber \ \cdot \frac{\prod\limits_{i<k}\,\,\prod\limits_{|\mathcal{S}_i|\ \textrm{even}}\,\,\prod\limits_{n_k=1}^\infty\,\prod\limits_{\stackrel{\textrm{\small{$\ldots,\,n_j ,\ldots=\, 1\,$}}}{j \in \mathcal{S}_i}}^\infty\,f \left(\left[\mathcal{S}_i \cup \{k\}\right] \circ x\right)}{\prod\limits_{i<k}\,\,\prod\limits_{|\mathcal{S}_i|\ \textrm{odd}}\,\,\prod\limits_{n_k=1}^\infty\,\prod\limits_{\stackrel{\textrm{\small{$\ldots,\,n_j ,\ldots=\, 1\,$}}}{j \in \mathcal{S}_i}}^\infty\,f \left(\left[\mathcal{S}_i \cup \{k\}\right] \circ x\right)}, \nonumber
\end{align}
where we have used Eq.~\ref{eq:fkcirc} to express $f_i$, for all $i<k$.  In the denominator, after unioning each odd cardinality set $\mathcal{S}_i\,, i<k$ with the set $\{k\}$,
we have all of the even cardinality sets in $\{\mathcal{S}_k\}$ in the denominator.  Similarly, we end up with  all of the odd cardinality sets in $\{\mathcal{S}_k\}$ in the numerator, including the singleton set $\{k\}$, which appears in the first factor.  This gives Eq.~\ref{eq:fkcirc}, completing the induction.
\end{proof}

\subsection{Simplification:  Consolidating Exponents}

We can simplify  Eq.~\ref{eq:fkcirc} by consolidating all indices $n_j, j \in \mathcal{S}_k$, into a single index $n$.  
\begin{lemma}
Let $\mathcal{S}_k \subset \mathbb{N^*} \,$with$\,\, \max(\mathcal{S}_k)=k$, where $k$ is a positive integer.  With the previous assumptions and notation, each component $f_k(x)$ of $f$ is given by
\begin{equation}
\label{eq:fkpsk}
f_k(x) = \lim\limits_{r \downarrow 1}\,\frac{\prod\limits_{|\mathcal{S}_k|\ \textrm{odd}}\,\mathcal{P}_f(\mathcal{S}_k, r, x)}{\prod\limits_{|\mathcal{S}_k|\ \textrm{even}}^{\ }\,\mathcal{P}_f(\mathcal{S}_k, r, x)},
\end{equation}
where $\mathcal{P}_f(\mathcal{S}_k, r, x)$ is defined by Equations~\ref{eq:factors}~and~\ref{eq:places}.
\end{lemma}
\begin{proof}
If we express Eq.~\ref{eq:fkcirc} explicitly,
\begin{equation}
\label{eq:fkmessy}
f_k(x) = \lim_{r\downarrow1}\,\frac{\prod\limits_{|\mathcal{S}_k|\ \textrm{odd}}\,\,\,\prod\limits_{\ldots,\,n_j,\ldots=\, 1\, }^\infty\,f\left(\left[\,\prod_{j\,\in\,\mathcal{S}_k}\,\frac{(r^j-1)^{1/j}}{r^{n_j}}\right]\,x\right)}{\prod\limits_{|\mathcal{S}_k|\ \textrm{even}}\,\,\,\prod\limits_{\ldots,\,n_j ,\ldots=\, 1\,}^\infty\,f\left(\left[\,\prod_{j\,\in\,\mathcal{S}_k}\,\frac{(r^j-1)^{1/j}}{r^{n_j}}\right]\,x\right)}.
\end{equation}
In both the numerator and denominator of $f_k(x)$, there is a product
\begin{align*}
&\prod_{\ldots,\,n_j,\,\ldots\, =\, 1}^\infty\,f\left(\,\left[\,\prod_{j\,\in\,\mathcal{S}_k}\,\frac{(r^j-1)^{1/j}}{r^{n_j}}\right]\,x\right)\,\,\,\,\,  = \prod_{\ldots,\,n_j,\,\ldots\, =\, 1}^\infty\,f\left(\,\left[\,\,\prod_{j\,\in\,\mathcal{S}_k}(r^j-1)^{1/j}\right]\,\cdot\,\frac{x}{{r^{\,\,\,\sum\limits_{\textrm{\scriptsize{$j\,\in\,\mathcal{S}_k$}}}\textrm{\footnotesize{$n_j$}}}}}\,\right)  \\
=\,&\,\,\,\,\prod_{n=|\mathcal{S}_k|}^\infty\,\left[\,f\left(\,\left[\prod_{j\in\mathcal{S}_k}\,(r^j-1)^{1/j}\right]\frac{x}{r^n}\right)\,\right]^{\,\,\textrm{\large{$\binom{n-1}{|\mathcal{S}_k|-1}$}}}\,\,\,=\,\,\,\,\,\prod_{n=|\mathcal{S}_k|}^\infty\,\,\mathcal{P}_f\left(\mathcal{S}_k, r, x\right)\,.
\end{align*}
The binomial coefficient $\binom{n-1}{|\mathcal{S}_k|-1}$ is the number of ways in which the $|\mathcal{S}_k|$ indices $\{\ldots,\,n_j,\,\ldots |\, j \in \mathcal{S}_k\}$, can add up to $n$ (Ref. \cite{composition}).  The index $n = \sum n_j$ starts at $|\mathcal{S}_k|$ because each $n_j$ is at least 1, and there are  $|\mathcal{S}_k|$ of them.  Substituting this result in both numerator and denominator of Eq.~\ref{eq:fkmessy} gives Eq.~\ref{eq:fkpsk}.
\end{proof}
To obtain $f(x)$, we multiply all $f_k(x)$ for $k\in\mathbb{N}^*$.  This effectively removes the restriction $\max(\mathcal{S}_k) = k$ in Eq.~\ref{eq:fkpsk}, giving us the Geometric Multiproduct for $f(x)$, Eq.~\ref{eq:powerproduct}.  This ends the proof of Theorem \ref{thm:GMP}.
\end{proof}
Mathematically, the Geometric Multiproduct does what Taylor series do.  It uses information about a function from an infinitesimally small neighborhood of one point to calculate the value of the function at another point.

\subsection{Practical Applications}

In engineering applications, to \emph{sample a signal} means to measure its magnitude periodically, e.g., at moments that constitute an arithmetic sequence in time, or at uniformly separated places.  The signal is often a voltage that represents a physical quantity, such as the pressure of a sound wave or the intensity of light at one point.  Periodic sampling of a signal is widely used in many modern inventions, such as digital telephone networks and digital cameras.  The samples are quantized and digitized (rounded to an integer and converted to a binary number) to facilitate computation and transmission.  Subsequently, the digitized samples can be converted back to an analog signal, e.g., speech or image.  In finance, recording the daily closing price of a stock is another example of sampling.

     Sampling has an important advantage over analytic methods.  If we can do all of our calculations with the digitized samples alone, we never need to determine the analytic form of the signal.  Ref.~\cite{dsp} describes how to process a signal by performing calculations with digitized arithmetic samples, for example, to remove unwanted frequency components in a digitized speech signal.

     The Geometric Multiproduct shares the same advantage.   Even if we don't know the analytic form of a signal, we can still make predictions of its future value, by using geometric samples:  measurements of the signal at points in time that form geometric sequences.  For this reason, it could potentially have many applications in engineering, science, finance and any other field in which forecasting plays an important role.





\begin{thebibliography}{00}
%
%
%
%
%

%
\bibitem{analyticlog}{ John B. Conway, Functions of One Complex Variable, Vol. 1, Graduate Texts in Mathematics, Springer-Verlag, New York, 2010, pp. 202-204}
%
%
\bibitem{composition}{Percy Alexander McMahon, Memoir on the Theory of the Compositions of Numbers,  Phil. Trans. Royal Society London, A \textbf{184}, 1883}
%
\bibitem{dsp}{Lawrence Rabiner, and Charles~M. Rader, Digital Signal Processing, John Wiley and Sons, New York, 1972}
\bibitem{infprod}{Remmert, Reinhold: Classical Topics in Complex Function Theory, Graduate Texts in Mathematics, Springer Verlag, New York, 1978, pp. 1-11}
\end{thebibliography}







\end{document}